\documentclass[12pt]{article}
\usepackage{amssymb,amsmath,amsthm}

\begin{document}

\title{Additive quaternary codes related to exceptional linear quaternary codes}
\author{J\"urgen Bierbrauer\\
Department of Mathematical Sciences\\
Michigan Technological University\\
Houghton, Michigan 49931 (USA) \\ \\
S. Marcugini and F. Pambianco \\
Dipartimento di Matematica e Informatica\\
Universit\`a degli Studi di Perugia\\
Perugia (Italy)
\footnote{This research was supported in part by Ministry for Education, University and Research of Italy (MIUR) and by the Italian National Group for Algebraic and Geometric Structures and their Applications (GNSAGA - INDAM).
}}

\maketitle
\newtheorem{Theorem}{Theorem}
\newtheorem{Proposition}{Proposition}
\newtheorem{Lemma}{Lemma}
\newtheorem{Definition}{Definition}
\newtheorem{Corollary}{Corollary}
\newtheorem{Example}{Example}
\def\nz{\mathbb{N}}
\def\gz{\mathbb{Z}}
\def\rz{\mathbb{R}}
\def\ef{\mathbb{F}}
\def\CC{\mathbb{C}}
\def\o{\omega}
\def\p{\overline{\omega}}
\def\e{\epsilon}
\def\a{\alpha}
\def\b{\beta}
\def\g{\gamma}
\def\d{\delta}
\def\l{\lambda}
\def\s{\sigma}
\def\bsl{\backslash}
\def\la{\longrightarrow}
\def\arr{\rightarrow}
\def\ov{\overline}
\def\sm{\setminus}
\newcommand{\D}{\displaystyle}
\newcommand{\T}{\textstyle}

\begin{abstract}
We study additive quaternary codes whose parameters are close to those
of the extended cyclic $[12,6,6]_4$-code or to
the quaternary linear codes
generated by the elliptic quadric in $PG(3,4)$ or its dual.
In particular we characterize those codes in the category of additive codes and
construct some additive codes whose parameters are better than those of any
linear quaternary code. 
\end{abstract}

\noindent {\bf Keywords}: Additive code, linear code.

\section{Introduction}
\label{introsection}

Additive codes are a far-reaching and natural generalization of linear codes,
see also Chapter 18 of \cite{book2nded}. Here we restrict to
the quaternary case which corresponds to choosing the alphabet as the Klein group $Z_2\times Z_2.$
The quaternary case is of special interest, among others because of a close link to the
theory of quantum stabilizer codes and their geometric representations, see~\cite{quantgeom,CRRS,DHY}.

\begin{Definition}
\label{addcodebasicdef}
Let $k$ be such that $2k$ is a positive integer.
An additive quaternary $[n,k]_4$-code ${\cal C}$ (length $n,$
dimension $k$) is a $2k$-dimensional subspace of $\ef_2^{2n},$
where the coordinates come in pairs.
We view the codewords as $n$-tuples where the coordinate entries are elements of $\ef_2^2.$
\par
A {\bf generator matrix} $G$ of ${\cal C}$ is a binary $(2k,2n)$-matrix whose
rows form a basis of the binary vector space ${\cal C}.$
\end{Definition}

\begin{Definition}
Let be ${\cal C}$ an additive $[n,k]_4$-code. The
{\bf weight} of a codeword is the number of its $n$ coordinates where
the entry is different from $00$. The minimum weight (equal to
minimum distance) $d$ of ${\cal C}$  is the smallest weight of its nonzero
codewords. The parameters are then also written  $[n,k,d]$. $A_i$ denotes the number of codewords of weight $i$.
The {\bf strength} of ${\cal C}$ is the largest number $t$ such that all
$(2k,2t)$-submatrices of a generator matrix whose columns
correspond to some $t$ quaternary coordinates have full rank $2t$.
\end{Definition}

The geometric interpretation of an additive $[n,k,d]_4$-code is a multifamily of
$n$ lines in $PG(2k-1,2)$ (the codelines) with the property that any hyperplane of $PG(2k-1,2)$ contains
at most $n-d$ of the codelines. Observe also that codelines may degenerate to points.

Code ${\cal C}$ has minimum distance $d$  if and
only if for each hyperplane $H$ of  $PG(2k-1,2)$ we find at least
$d$ codelines (in the multiset sense), which are not contained in $H$. Strength $t$ means that any set of $t$ codelines is in general position.

As a line not intersecting a hyperplane meets it in a point, it is natural to consider also the
following generalization which may be called mixed quaternary-binary codes:

\begin{Definition}
\label{nmsetdef}
An $(n,m)$-set in $PG(k-1,2)$ is a set of $n$ lines and $m$ points.  It has strength $t$ if
any $t$ of those codeobjects (lines or points) are in general position.
\end{Definition}

A basic theorem by Blokhuis-Brouwer \cite{BB} characterizes the linear among the additive
quaternary codes (see also \cite{book2nded}, Theorem 18.48).
It states that a family of codelines in $PG(2k-1,2)$ generates an $\ef_4$-linear
code if and only if each subset of codelines generates a vector space of even binary dimension.
Among the exceptional linear quaternary codes are the hexacode $[6,3,4]_4$
(geometrically the hyperoval in $PG(2,4),$ see\cite{book2nded}, Section 17.1 and Exercise 3.7.4), the extended cyclic code $[12,6,6]_4$ (see \cite{book2nded}, Section 13.4 and Exercise 17.1.10) and the dual pair of codes
$[17,4,12]_4$ and $[17,13,4]_4$ described by the elliptic quadric in $PG(3,4)$ (see Section 17.3 of \cite{book2nded}).
In the present paper we study the generalization from linear quaternary to additive codes of code parameters
related to those of our exceptional linear codes.
The hexacode case is not interesting. In fact, the parameters $[6,3,4]_4$ show immediately
that any three of the six codelines describing this additive code are in general position. This implies that the
additive code is in fact linear, see \cite{BB}.
Additive $[12,6,6]_4$-codes are studied in Section \ref{1266section}.
We show that the linear $[12,6,6]_4$-code is uniquely determined and that each additive
$[12,6,6]_4$-code has one of two weight distributions.  Either it has the weight distribution of the linear
code or it is a quantum $[[12,0,6]]_2$-code. Such a quantum code is uniquely determined.
It is known as the dodecacode, see \cite{CRRS}.
Section \ref{quadriccharactsection} presents characterizations of the elliptic quadric in $PG(3,4).$
We show that additive $[17,4,12]_4$-codes and additive $[17,13,4]_4$-codes are necessarily $\ef_4$-linear.
\par
In the final Section \ref{lastsection} we consider additive codes of larger length
with parameters related to those of the elliptic quadric.
The nonexistence of an additive $[18,14,4]_4$-code is an
easily obtained consequence
of the characterization of the elliptic quadric. In other words, the additive relaxation of an $18$-cap in
$PG(3,4)$ does not exist. Finally we construct $[22,17.5,4]_4$-codes.
From them, $[19+i,14.5+i,4]_4$-codes $i=0, \dots, 2$ can be obtained by shortening. Such additive codes are new and their parameters cannot be reached by linear codes.
For the parameters of optimal additive codes of small dimension see \cite{BBFMP} and the references therein.
For simplicity we denote $PG(i-1,2)$ also as $V_i.$

\section{Additive $[12,6,6]_4$-codes}
\label{1266section}

\begin{Theorem}
\label{1266lineartheorem}
A linear $[12,6,6]_4$-code is uniquely determined.
\end{Theorem}
\begin{proof}
Let us start using geometric arguments. The ambient space is $PG(5,4),$ each
hyperplane contains at most $6$ codepoints, each secundum $PG(3,4)$ contains
at most $4$ codepoints. This shows that any four codepoints do generate a secundum
and any five codepoints generate a hyperplane. In particular the strength is $5,$ in other words:
the dual code has the same parameters $[12,6,6]_4.$
\par
Let $S$ be a secundum containing $4$ codepoints. There are two possibilities for the
distribution of codepoints on the hyperplanes containing $S$.
Either two of those hyperplanes contain one more codepoint, the others contain
two more codepoints or one of the $5$ hyperplanes containing $S$ has no further codepoint,
the four remaining ones contain two more codepoints each.
We try to construct a generator matrix $G=\left(\begin{array}{c|c}
I & P
\end{array}\right)$ where $P=(p_{ij})$ is a $(6,6)$-matrix.
Denote by $z_i$ the rows of the generator matrix, by $v_i,i=1,\dots ,6$ the rows of $P.$
\par
Let us consider the case $A_7=0,$ equivalently no
hyperplane contains precisely $5$ codepoints, again equivalently
the hyperplanes containing 6 codepoints form the blocks of a Steiner system $S(5,6,12).$

The usual arguments show that we can choose

$$P=\left(\begin{array}{c|c|c|c|c|c}
0 & 1  & 1  & 1   & 1 & 1  \\
1 & 0  &  & &  &   \\
1 &     & 0  &  &   &   \\
1 &     &  & 0 &  &  \\
1 &     &  &   & 0 &   \\
1 &     &  &  &  &  0 \\
\end{array}\right).$$

Codewords $\l z_1+z_2$ show that each (nonzero) entry occurring among the
$p_{23},p_{24},p_{25},p_{26}$ occurs precisely twice. The same holds for each of the
rows $v_i,i>1.$
We can choose

$$P=\left(\begin{array}{c|c|c|c|c|c}
0 & 1  & 1  & 1   & 1 & 1  \\
1 & 0  &  a & a & b &  b \\
1 &     & 0  &  &   &   \\
1 &     &  & 0 &  &  \\
1 &     &  &   & 0 &   \\
1 &     &  &  &  &  0 \\
\end{array}\right)$$

\noindent where $a,b$ are different nonzero elements. Clearly we can choose notation such that
the corresponding nonzero entries in $v_3$ are the same, $a$ and $b.$
We must have $p_{35}\not=p_{36}$ as otherwise there is a linear relation involving the
last two columns of $P$ and the three last columns of the unit matrix. We can choose

$$P=\left(\begin{array}{c|c|c|c|c|c}
0 & 1  & 1  & 1   & 1 & 1  \\
1 & 0  &  a & a & b &  b \\
1 &     & 0  &  & a  &  b \\
1 &     &  & 0 &  &  \\
1 &     &  &   & 0 &   \\
1 &     &  &  &  &  0 \\
\end{array}\right).$$

The codeword $z_2+z_3$ shows $p_{34}=b.$ We have

$$P=\left(\begin{array}{c|c|c|c|c|c}
0 & 1  & 1  & 1   & 1 & 1  \\
1 & 0  &  a & a & b &  b \\
1 &  a  & 0  & b & a  &  b \\
1 &     &  & 0 &  &  \\
1 &     &  &   & 0 &   \\
1 &     &  &  &  &  0 \\
\end{array}\right).$$

Now
$$\o z_2+z_3=(0,\o ,1,0,0,0,\p ,a,\o a,\o a+b,\o b+a,\p b).$$
Exactly one of $\o a+b$ and $\o b+a$ is nonzero. This shows that we have a
codeword of weight $7,$ contradiction.
 \par
We have $A_7\not=0,$ equivalently there is a
hyperplane with precisely $5$ codepoints, equivalently there is a secundum for which the
first of the two possibilities above is satisfied. Let us choose $S=(x_1=x_2=0)$ as this secundum.
The two hyperplanes containing $S$ and containing $5$ codepoints each are
$H_1=(x_1=0)$ and $H_2=(x_2=0).$ This means that $v_1,v_2$ have no entry $0.$
Comparing $z_1$ and $z_2$ we see that we can choose
$v_1=(1,1,1,1,1,1), v_2=(1,1,\o ,\o ,\p ,\p ).$
Let $\nu$ be the total number of entries $0$ in $P.$ As each row of $P$ has at most one
such entry it follows $\nu\leq 4.$
Assume at first $\nu\leq 3.$ Choose notation such that $v_3$ has no entry $0.$
Codewords $z_3+\l z_1$ show that each nonzero element occurs precisely twice in $v_3.$
We can choose $p_{31}=1.$
As $d(v_2,v_3)=4$ we can choose notation such that $p_{33}=\o .$
Strength $5$ shows $p_{32}\not=1 .$ Comparison with $v_2$ shows
$p_{32}\not=\o $ (as otherwise $v_2,v_3$ would agree in one of the last
two coordinates). It follows $p_{32}=\p  $ and $p_{34}=\p .$
By symmetry we can choose $p_{35}=1, p_{36}=\o .$ We have

$$P=\left(\begin{array}{c|c|c|c|c|c}
1 & 1 & 1 & 1 & 1 & 1  \\
1 & 1 & \o & \o & \p & \p  \\
1  & \p & \o & \p & 1 &  \o \\
  &  &  &  &  &  \\
  &  &  &  &  &  \\
  &  &  &  &  &  \\
\end{array}\right).$$

Observe that rows $v_1,v_2,v_3$ of $P$ form a generator matrix of the hexacode $[6,3,4]_4.$
Assume now $\nu\leq 2$ and choose notation such that $v_4$ has weight $6, p_{41}=1 .$
Strength $5$ shows $p_{42}\not=1 , p_{45}\not=1, p_{43}\not=\o  .$
Assume $p_{43}=1.$ Then $p_{46}=\p , p_{45}=\o$ and
$v_4=(1,\o ,1,\p ,\o ,\p ).$ However $v_1+v_4+\o (v_2+v_3)=0$ yields a
codeword of weight $4.$
We have $p_{43}=\p .$ Again $v_4$ is then uniquely determined,
$v_4=(1,\p ,\p ,\o ,\o ,1)$ and $\o (v_1+v_2)+v_3+v_4=0,$ contradiction.
It follows $\nu =3$ or $\nu =4.$
\par
Assume at first $\nu =3.$
We have the first three rows of $P$ as before (forming a generator matrix $G$ of the hexacode),
and each subsequent row of $P$ contains one entry $0.$
Denote the columns of $G$ by $s_1,\dots ,s_6,$ where $s_6=\tilde{s}_4$ and  $s_1=\tilde{s}_1.$
Consider transformations $g(\b ,\g )$ defined by
$$v_1\mapsto v_1, v_2\mapsto\b v_2, v_3\mapsto\g v_3.$$
Then $g(1,\o )$ maps
$$(s_1,\dots ,s_6)\mapsto (\tilde{s}_2,s_1,s_4,\tilde{s}_5,s_6,s_4).$$
After taking conjugates and rearranging columns this leads back to $G.$
This maps a candidate row $w=(a,b,c,d,e,f)$ to $(\tilde{b},\tilde{a},\tilde{e},\tilde{f},\tilde{d},\tilde{c}).$
This describes a permutation of order $4.$ In the same manner
$g(\p ,\p )$ describes the involution $w=(a,b,c,d,e,f)\mapsto (\tilde{c},\tilde{d},\tilde{a},\tilde{b},\tilde{f},\tilde{e}).$
Consider $(g(1,\o )g(\p ,\p ))^2.$ It maps $w\mapsto (d,c,f,e,b,a),$ a permutation of order $6.$
This shows that we can choose $p_{41}=0.$
\par
The permutation $\rho :v_1\mapsto v_2\mapsto v_3\mapsto v_1$ maps
$w=(a,b,c,d,e,f)\mapsto (a,\p c,\o e,\p d,b,\o f)$
and the involution $\s :v_1\mapsto v_1, v_2\mapsto v_3\mapsto v_2$ maps
$w\mapsto (a,e,c,f,b,d).$ Clearly $\langle\rho ,\s\rangle\cong S_3.$
\par
A direct check shows that there are only three vectors $w$ with first entry $0,$
one entry $1$ and two entries $\o$ as well as two entries $\p$ which may be used as $v_4$
(such that the corresponding four-dimensional code has minimum distance $6$).
Those are
$$w=(0,\o ,1,\p ,\o ,\p ), (0,\o ,\p ,1 ,\p ,\o ), (0,\p ,\p ,\o ,\o ,1 ).$$
They form an orbit under the group $S_3$ above. It follows that we may choose
$v_4$ as the first of those vectors. We have

$$P=\left(\begin{array}{c|c|c|c|c|c}
1 & 1 & 1 & 1 & 1 & 1  \\
1 & 1 & \o & \o & \p & \p  \\
1  & \p & \o & \p & 1 &  \o \\
0  & \o & 1  & \p & \o &  \p \\
  &  &  &  &  &  \\
  &  &  &  &  &  \\
\end{array}\right).$$

A little computer program shows that this cannot be completed.
\par
Assume finally that $\nu =4.$ The first two rows of $P$ are as before.
Clearly we can assume $p_{31}=0.$
Each of the rows $v_i,i>2$ contains one nonzero field element with
frequency $1,$ each of the others with frequency $2.$ By an obvious
operation we can assume that the field element $1$ occurs with
frequency $1$ in each $v_i,i>2.$
It is impossible that $p_{32}=1$ as in this case all entries of
$v_2+v_1$ are $0$ or $1$ which yields a contradiction.
Conjugation shows that we may choose
$p_{32}=\o .$
Assume $p_{35}=1.$ Then $v_3=(0,\o ,\o ,\p ,1,\p )$ which yields
$wt(\o v_2+v_3)=3,$ contradiction. We can choose $p_{33}=1$
and $v_3=(0,\o ,1,\p ,\o ,\p ).$
We have

 $$P=\left(\begin{array}{c|c|c|c|c|c}
1 & 1 & 1 & 1 & 1 & 1  \\
1 & 1 & \o & \o  & \p & \p  \\
0  & \o & 1 & \p & \o &  \p \\
  &  &  &  &  &   \\
  &  &  &  &  &  \\
  &  &  &  &  &  \\
\end{array}\right).$$

A computer program shows that this can be completed in precisely
$12$ ways. The resulting matrices have their zeroes in the last three rows
in columns either $2,3,4$ or $2,5,6.$
The linear mapping $\tau :v_1\mapsto v_2\mapsto v_1, v_3\mapsto v_3$ induces
$\tau (w)=(a,b,\o f,\o e,\p d,\p c).$ We can therefore assume that the
zeroes occur in positions $2,3,4$ in the last three rows of $P.$
Those six matrices generate equivalent codes. We conclude that the
linear $[12,6,6]_4$-code is uniquely determined. Here is one example:

 $$P=\left(\begin{array}{c|c|c|c|c|c}
1  & 1 & 1 & 1 & 1 & 1  \\
1   & 1   & \o  & \o  & \p & \p  \\
0   & \o  & 1   & \p & \o  &  \p \\
\p  &  0  &  \p & \o & \o &  1 \\
\p  &  \o & 0   & \o & 1  &  \p \\
\o  & 1   &  \p & 0  &  \o &  \p \\
\end{array}\right).$$
\end{proof}

The weight distribution of this uniquely determined quaternary linear $[12,6,6]_4$-code is

$$A_0=1, A_6=330, A_7=396, A_{8}=495, A_{9}=1320, $$
$$A_{10}=990, A_{11}=396, A_{12}=168.$$

Let us consider additive codes with these parameters.
In fact, we work with the dual, a $[12,6]_4$-code of strength $5.$
Projection to some $7$ codelines generating the ambient space yields a $[7,6]$-code of strength $5.$
Its dual is a $[7,1,6]$-code. Clearly those codes can be characterized.
There may be a 0-coordinate or up to three points. This leads to the following
list of $[7,1,6]$-codes.
In each of these 5 cases we determined the dual
and the sets of $5$ lines in $PG(11,2)$
which complete this to a code of strength $5.$
The first configuration

$$\left(\begin{array}{c|c|c|c|c|c|c}
L_1 & L_2 & L_3 & L_4 & L_5 & L_6 & L_7 \\
10  & 10 & 10 & 10 & 10 & 10 & 00 \\
01  & 01 & 01 & 01 & 01 & 01 & 00 \\
\end{array}\right)$$

\noindent leads to completions with two different weight distributions:
$${\bf quantum:} \ \ A_0=1, A_6=396, A_8=1485, A_{10}=1980, A_{12}=234;$$
$${\bf linear:}\ \ A_0=1, A_6=330, A_7=396, A_{8}=495, A_{9}=1320, $$
$$A_{10}=990, A_{11}=396, A_{12}=168.$$
The latter is the weight distribution of the linear code.
Refer to these weight distributions as the quantum case and
the quaternary linear case, respectively.
The second configuration

$$\left(\begin{array}{c|c|c|c|c|c|c}
L_1 & L_2 & L_3 & L_4 & L_5 & L_6 & L_7 \\
10  & 10 & 10 & 10 & 10 & 10 & 00 \\
01  & 01 & 01 & 01 & 01 & 01 & 01 \\
\end{array}\right)$$

\noindent and the third configuration

$$\left(\begin{array}{c|c|c|c|c|c|c}
L_1 & L_2 & L_3 & L_4 & L_5 & L_6 & L_7 \\
10  & 10 & 10 & 10 & 10 & 01 & 00 \\
01  & 01 & 01 & 01 & 01 & 00 & 01 \\
\end{array}\right)$$

\noindent do not produce codes. The fourth configuration

$$\left(\begin{array}{c|c|c|c|c|c|c}
L_1 & L_2 & L_3 & L_4 & L_5 & L_6 & L_7 \\
10  & 10 & 10 & 10 & 01 & 01 & 00 \\
01  & 01 & 01 & 01 & 01 & 00 & 01 \\
\end{array}\right)$$

\noindent produces a code with the weight
distribution of the quantum case, the fifth configuration

$$\left(\begin{array}{c|c|c|c|c|c|c}
L_1 & L_2 & L_3 & L_4 & L_5 & L_6 & L_7 \\
10  & 10 & 10 & 10 & 10 & 10 & 10 \\
01  & 01 & 01 & 01 & 01 & 01 & 01 \\
\end{array}\right)$$

\noindent yields  codes with
the weight distribution of the linear case.
Observe that we do not have a complete classification.
Our result is as follows:

\begin{Theorem}
\label{1266addtheorem}
Each additive $[12,6,6]_4$-code
has either the weight distribution of the linear code
or the weight distribution of the quantum case.
\end{Theorem}

Each additive  $[12,6,6]_4$-code having the weight
distribution of the quantum case is in fact a
$[[12,0,6]]_2$ quantum code, equivalently
it is self-dual with respect to the symplectic form.
It is known that this quantum code is uniquely determined:
the dodecacode.
It is not clear if
all additive  $[12,6,6]_4$-codes having the weight
distribution of the linear case are necessarily $\ef_4$-linear.

\section{Characterizations of the elliptic quadric}
\label{quadriccharactsection}

The elliptic quadric in $PG(3,4)$ is the uniquely determined $17$-cap in $PG(3,4).$
Using representatives of those points as columns
of a $(4,17)$-matrix with entries in $\ef_4,$ we obtain a generator matrix $M$ of a linear
$[17,4,12]_4$-code whose dual is a $[17,13,4]_4$-code. Both those codes have optimal parameters.
As $17$ is the largest size of a cap in $PG(3,4),$ there is no linear $[18,14,4]_4$-code.
The elliptic quadric has a large sphere of influence. In fact, linear codes
$[16,12,4]_4$ and $[15,11,4]_4$ correspond to caps of sizes $15$ and $14$
respectively, and it is known that such caps are contained in the elliptic quadric
(see \cite{cap4}). It follows that linear codes with those parameters are uniquely determined.
By what we observed above it will be sufficient to prove that such additive codes are
necessarily $\ef_4$-linear. Here the characterization from \cite{BB}
is useful again.
Observe also that concatenation of an additive $[n,k,d]_4$-code with the binary linear
$[3,2,2]_2$-code yields a binary linear $[3n,2k,2d]_2$-code. This fact can be used to
obtain non-existence results.

\begin{Theorem}
\label{17412theorem}
An additive $[17,4,12]_4$-code is uniquely determined.
\end{Theorem}
\begin{proof}
Assume at first that one of the $17$ codeobjects is a point.
This yields a subcode with parameters $[16,3.5,12]_4.$
Concatenation yields a binary $[48,7,24]_2$-code which however cannot exist
by the Griesmer bound. An analogous argument shows that, because of the non-existence
of a $[15,2.5,12]_4$-code, it is impossible that two codelines intersect.
It follows that the $17$ codelines form a partial spread of lines.
Let $M$ be the union of the points on codelines. Denote the
elements of $M$ by codepoints.
\par
Observe that the ambient space is a $V_8.$ The definition of the distance shows that
each hyperplane $V_7$ contains at most  $5\times 2+17=27$ codepoints.
Obvious counting arguments imply that each secundum $V_6$ has at most $15$ codepoints,
a $V_5$ has at most $9$ codepoints,
a $V_4$ has at most $6$ codepoints and
a $V_3$ has at most $4$ codepoints.
This implies that our code has strength $3$: any three codelines are in
general position. In fact, assume some three codelines are in a $V_5.$ Then the third of those lines
meets the $V_4$ generated by the first two in a point. We obtain
$7$ codepoints in a $V_4,$ contradiction.
\par
Let now $U$ be a plane $V_3$ containing a
codeline $L_1$ and an isolated codepoint on codeline $L_2.$
Consider the factor space $\Pi ,$  a $PG(4,2).$
For each point $P$ of $\Pi ,$ let the weight $w(P)$
be the number of codepoints in the preimage of $P$ which are outside $U.$
For each subspace $X$ of $\Pi$ let $w(X)=\sum_{P\in W}w(P).$
The distribution of codepoints on subspaces shows that $w(P)\leq 2$ for each
point $P$ of $\Pi .$ Also $w(\Pi )=47$ and each line $l$ has weight $w(l)\leq 5.$
Considering a point $P$ of weight $2$ and the lines of $\Pi$ passing through it, we see
that all weights of points are $1$ or $2$ and each line through $P$ contains precisely
one point of weight $1.$ It follows that $\Pi$ possesses a hyperplane $H_0$ all of whose
points have weight $1,$ where all affine points (outside $H_0$) have weight $2.$
\par
Observe that $\Pi$ has a special point $P_0$ of weight $2$ whose preimage contains
codeline $L_2.$ Because of strength $2$ each of the remaining $15$ codelines
(different from $L_1,L_2$) has as image a line of $\Pi .$
Denote those images by $M_1,\dots ,M_{15}.$ We have that $P_0$ is not contained
in any of those lines of $\Pi ,$ each point of $H_0$ is on precisely one of those
lines and each affine point different from $P_0$ is on two of those lines.
\par
Choose coordinates: let $v_1,\dots ,v_8$ be a basis of the ambient space,
$U=\langle v_1,v_2,v_3\rangle$ and use coordinates $y_1,\dots ,y_5$ for the
actor space $\Pi .$ As hyperplane choose $H_0=(y_1=0)$ and as special point
 $P_0=10000.$
An obvious calculation with $(5,5)$-matrices shows that
the first two lines $M_1,M_2$ can be chosen as $\langle 01000,10100\rangle$ and
$\langle 00010,10001\rangle .$ This leads to the first four lines in $PG(7,2)$ as follows:

$$\left(\begin{array}{c|c|c|c}
L_1 & L_2 & L_3 & L_4  \\
10 & 00 & 00 & 00   \\
01 & 00 & 00 & 00   \\
00 & 10 & 00 & 00   \\ \hline
00 & 01 & 01 & 01   \\
00 & 00 & 10 & 00  \\
00 & 00 & 01 & 00  \\
00 & 00 & 00 & 10    \\
00 & 00 & 00 & 01  \\
\end{array}\right)$$

This leads to the problem of finding a family $M_1,\dots ,M_{15}$ of lines in $\Pi$ such that
$M_1,M_2$ are as above, $P_0$ is on none of those lines, each of the remaining points
of weight $2$ is on two of the $M_i$ and each point of the hyperplane $H_0$ is on
precisely one of the $M_i.$
A computer program revealed that this problem in $PG(4,2)$ has precisely
246 different solutions. The completion to the set of 17 codelines succeeds only in one
case, the case of the elliptic quadric.
\end{proof}

\begin{Theorem}
\label{17134theorem}
An additive $[17,13,4]_4$-code is uniquely determined.
\end{Theorem}
\begin{proof}
The dual of our additive $[17,13,4]_4$-code is a
$[17,4]$-code of strength $3.$ Because of Theorem \ref{17412theorem}
we can assume that it has distance $\leq 11,$ meaning
that at least some $6$ codelines are on a hyperplane.
\par
Assume there are $7$ lines on a hyperplane.
This would yield a $(7,11)$-set of strength $3$ in $PG(6,2)$ (see Definition \ref{nmsetdef}).
We saw in \cite{noquant1354} that such a set does not exist.
It follows that we must have a $(6,11)$-set of strength $3$ on a hyperplane $H.$
Those sets were determined by a computer program. There are two families of $6$ lines each, the hexacode
$[6,3,4]_4$
and another code. The hexacode yields $6$ classes of $(6,11)$-sets, the other
family of six lines yields a unique $(6,11)$-set. An exhaustive search showed that
none of those sets could be completed to a set of $17$ lines of strength $3$.
This shows that a $[17,13,4]_4$-code is necessarily the dual of a $[17,4,12]_4$-code.
\end{proof}

\section{Additive codes of larger length}
\label{lastsection}

\begin{Theorem}
\label{no18144theorem}
There is no additive $[18,14,4]_4$-code.
\end{Theorem}
\begin{proof}
 Consider the dual, an additive $[18,4]_4$-code
of strength $3.$ The first $17$ of those lines describe a
$[17,4]_4$-code of strength $3.$ By Theorem \ref{17134theorem}
this is the dual of the elliptic quadric code. It suffices therefore to
start from this set of $17$ lines and to check by a simple computer
program that there is no line in $PG(7,2)$ which completes it to a
$[18,4]_4$-code
of strength $3.$
\end{proof}

Finally we construct an example of a $[22,17.5,4]_4$-code.
In fact we work with the dual, a $[22,4.5]_4$-code of strength $3.$
We work in $PG(8,2).$ The underlying vector space is $V(9,2)$ with
basis $e_1,\dots ,e_9.$ Our code consists of the lines
$L_1,\dots ,L_{22}.$ It contains two copies of the hexacode,
one consisting of
$$L_1=\langle e_1,e_2\rangle , L_2=\langle e_3,e_4\rangle , L_3=\langle e_5,e_6\rangle ,
L_4=\langle e_2+e_3+e_6,e_1+e_2+e_4+e_5\rangle ,$$
$$L_5=\langle e_2+e_4+e_5+e_6,e_1+e_3+e_5\rangle , L_6=\langle e_2+e_3+e_4+e_5 ,e_1+e_4+e_6\rangle$$
forming the hexacode in the subspace $\langle e_1,\dots ,e_6\rangle ,$
the second consisting of
$$L_1,L_2,L_7=\langle e_2+e_8,e_3+e_4+e_7+e_8\rangle,L_8=\langle e_3+e_4+e_8,e_1+e_3+e_7\rangle ,$$
$$L_9=\langle e_1+e_2+e_7,e_1+e_3+e_8\rangle ,L_{10}=\langle e_4+e_7,e_1+e_2+e_7+e_8\rangle$$
in the space $\langle e_1,e_2,e_3,e_4 ,e_7,e_8\rangle .$
Let $H$ be the hyperplane $\langle e_1,\dots ,e_8\rangle$ generated by the hexacodes. One further codeline
$L_{11}=\langle e_2+e_3+e_6+e_8,e_3+e_4+e_5+e_7\rangle$ is contained in $H,$ the remaining codelines are

$$\left(\begin{array}{c|c|c|c|c|c|c|c|c|c|c}
L_{12} & L_{13} & L_{14} & L_{15} & L_{16} & L_{17} & L_{18} & L_{19} & L_{20} & L_{21} & L_{22} \\
01 & 00 & 00 & 10 & 00  & 10  & 10  & 10  & 01  & 10 & 01 \\
11 & 01 & 01 & 00 & 00  & 00  & 00  & 11  & 10  & 10 & 01 \\
01 & 01 & 00 & 10 & 01  & 11  & 01  & 10  & 01  & 10 & 10  \\
10 & 00 & 00 & 10 & 11  & 10  & 10  & 00  & 10  & 00 & 00 \\
10 & 11 & 10 & 11 & 10  & 11  & 00  & 11  & 10  & 10 & 00 \\
10 & 10 & 11 & 01 & 00  & 10  & 10  & 11  & 00  & 01 & 10 \\
00 & 10 & 01 & 01 & 01  & 11  & 00  & 10  & 11  & 10 & 11\\
10 & 01 & 10 & 10 & 10  & 10  & 11  & 10  & 00  & 01 & 01 \\ \hline
01 & 01 & 01 & 01 & 01  & 01  & 01  & 01  & 01  & 01 & 01 \\
\end{array}\right).$$

\end{document}